\documentclass[11pt]{amsart}
\usepackage{amssymb}
\usepackage{amsmath}
\usepackage[active]{srcltx}
\usepackage{t1enc}
\usepackage[latin2]{inputenc}
\usepackage{verbatim}
\usepackage{amsmath,amsfonts,amssymb,amsthm}
\usepackage[mathcal]{eucal}
\usepackage{enumerate}
\usepackage[centertags]{amsmath}
\usepackage{graphicx}
\graphicspath{ {./images/} }

\newtheorem{theorem}{Theorem}

\newtheorem{lemma}{Lemma}
\newtheorem{remark}{Remark}
\newtheorem{definition}{Definition}
\newtheorem{corollary}{Corollary}

\begin{document}
\author{N. Nadirashvili, G. Tephnadze and  G. Tutberidze}
\title[Approximate Identity and Fejér means]{Almost everywhere and norm convergence of Approximate Identity and Fejér means of trigonometric and Vilenkin systems}

\address{N. Nadirashvili, The University of Georgia, School of science and
	technology, 77a Merab Kostava St, Tbilisi 0128, Georgia.}
\email{nato.nadirashvili@ug.edu.ge \ \ \ \ nato.nadirashvili@gmail.com  }
\address {G. Tephnadze, The University of Georgia, School of Science and Technology, 77a Merab Kostava St, Tbilisi, 0128, Georgia.}
\email{g.tephnadze@ug.edu.ge}
\address{G.Tutberidze, The University of Georgia, School of science and technology, 77a Merab Kostava St, Tbilisi 0128, Georgia.}
\email{g.tutberidze@ug.edu.ge, \ \ \ \ giorgi.tutberidze1991@gmail.com}
\thanks{The research was supported by Shota Rustaveli National Science Foundation grant FR-19-676.}

\date{}
\maketitle

\begin{abstract}
In this paper, we investigate very general approximation kernels with special properties, called an approximate identity, and prove almost everywhere and norm convergence of these general methods,  which consists of a class of summability methods and provide norm and a.e. convergence of these summability methods with respect to the trigonometric system. Investigations of these summations can be used to obtain norm convergence of Fej\'er means with respect to the Vilenkin system also, but these methods are not useful to study a.e. convergence in this case, because of some special properties of the kernels of Fej\'er means.
Despite these different properties we give alternative methods to prove almost everywhere convergence of  Fejér means with respect to the Vilenkin systems. 
\end{abstract}

\date{}



\section{Introduction}

Let define Fourier coefficients,
partial sums, Fejér means and  kernels
with respect to the Vilenkin and trigonometric systems  of any integrable function  in the usual manner: 
\begin{eqnarray*}
	\widehat{f}^w(k) &:&=\int f\overline{w}_{k}d\mu \text{\thinspace
		\qquad\ \ \ \ }\left(k\in \mathbb{N}, \ \ w=\psi \ \ \text{or} \ \ w=T\right), \\
	S^w_{n}f &:&=\sum_{k=0}^{n-1}\widehat{f}\left( k\right) \psi _{k}\ \text{
		\qquad\ \ }\left(n\in \mathbb{N}_{+},\text{ }S_{0}f:=0, \ \  w=\psi \ \ \text{or} \ \ w=T \right), \\
	\sigma^w _{n}f &:&=\frac{1}{n}\sum_{k=0}^{n-1}S^w_{k}f\text{ \qquad\ \ \ \ \ }
	\left( \text{ }n\in \mathbb{N}_{+}\text{ }\right), \\
	K^w_{n} &:&=\frac{1}{n}\overset{n-1}{\underset{k=0}{\sum }}D^w_{k}\text{ \qquad\ \ \ \ \
		\ \ \thinspace }\left( \text{ }n\in \mathbb{N}_{+},\text{ }\ \  w=\psi \ \ \text{or} \ \ w=T \right),
\end{eqnarray*}
where  $\mathbb{N}_{+}$ denote the set of the positive integers, $\mathbb{N}:=
\mathbb{N}_{+}\cup \{0\}.$

It is well-known (for details see e.g books \cite{AVD}, \cite{Garsia}  and \cite{Tor1}) that  Fejér means 

$$\sigma^w_nf \ \left( w=\psi \  \text{or} \  w=T\right),$$
where $\sigma_n^\psi$ and $\sigma_n^T$ are Vilenkin-Fejér and trigonometric-Fejér means, respectively, converge to the function $f$ in $L_p$ norm, that is 

\begin{equation*}
\left\Vert \sigma^w_nf-f\right\Vert_p\to 0, \ \ \text{as} \ \ n\to\infty, \ \ \left( w=\psi \  \text{or} \  w=T\right) 
\end{equation*}
for any $f\in L_p$  where $1\leq p< \infty.$  Moreover, (see e.g. \cite{BN} and  \cite{BNT}) if we consider the maximal operator of Fejér means with respect to Vilenkin and trigonometric systems defined by

\begin{eqnarray*}
	\sigma ^{\ast,w}f&:=&\sup_{n\in\mathbb{N}}\left\vert \sigma^w_nf\right\vert \ \left( w=\psi \  \text{or} \  w=T\right),
\end{eqnarray*}
then the weak type inequality 

\begin{equation*}
\mu \left( \sigma ^{*,w}f>\lambda \right) \leq \frac{c}{\lambda }\left\|
f\right\| _{1},\text{ \qquad }\left(f\in L_1(G_m), \ \  \lambda >0\right)
\end{equation*}
was proved in Zygmund \cite{Zy} for the trigonometric series, in Schipp
\cite{Sc} for Walsh series and in Pál, Simon \cite{PS} (see also \cite{PTW}, \cite{tepthesis1}, \cite{We1,We2}) for bounded Vilenkin system. It follows that  the F\'ejer means with respect to trigonometric and Vilenkin systems of any integrable function converges a.e. to this function.

In the books \cite{Garsia}, \cite{MS1} and \cite{tepthesis} was investigated very general approximation kernels with special properties, called an approximate identity which consists of a class of summability methods such us Fejér means.

In this paper  we investigate more general summability methods which are called the approximation identities, which consist of a class of summability methods and provide norm and a.e. convergence of these summability methods with respect to the trigonometric system. Investigations of these summations can be used to obtain norm convergence of Fej\'er means with respect to the Vilenkin system also, but these methods are not useful to study a.e. convergence in this case, because of some special properties of the kernels of Vilenkin-Fej\'er means.
Despite these different properties, we give alternative methods to prove almost everywhere convergence of  Fejér means with respect to the Vilenkin systems. 

This paper is organized as follows: in order not to disturb our discussions
later on some definitions and notations are presented in Sections 2 and 3. Moreover, For the proofs of the main results we need some auxiliary Lemmas, some of them are new and of independent interest. These results are also presented in Sections 2 and 3. The main result with proof is given in Sections 4 and 5. \qquad

\section{Fejér means with respect to the Vilenkin systems }

Let $m:=(m_{0},m_{1},\dots)$ denote a sequence of the positive integers not
less than 2. Denote by 
\begin{equation*}
Z_{m_{k}}:=\{0,1,\dots,m_{k}-1\}
\end{equation*}
the additive group of integers modulo $m_{k}.$

Define the group $G_{m}$ as the complete direct product of the group $%
Z_{m_{j}}$ with the product of the discrete topologies of $Z_{m_{j}}$ $^{,}$%
s. In this paper we discuss bounded Vilenkin groups only, that is $\sup_{n\in \mathbb{N}}m_{n}<\infty .$

The direct product $\mu $ of the measures 
\begin{equation*}
\mu _{k}\left( \{j\}\right):=1/m_{k}\text{ \qquad }(j\in Z_{m_{k}})
\end{equation*}
is the Haar measure on $G_{m_{\text{ }}}$with $\mu \left( G_{m}\right) =1.$

The elements of $G_{m}$ are represented by the sequences 
\begin{equation*}
x:=(x_{0},x_{1},\dots,x_{k},\dots)\qquad \left( \text{ }x_{k}\in
Z_{m_{k}}\right).
\end{equation*}

It is easy to give a base for the neighbourhood of $G_{m}$ namely
\begin{equation*}
I_{0}\left( x\right):=G_{m}, \ \ 
I_{n}(x):=\{y\in G_{m}\mid y_{0}=x_{0},\dots,y_{n-1}=x_{n-1}\}\text{ }(x\in
G_{m},\text{ }n\in \mathbb{N})
\end{equation*}%
Denote $I_{n}:=I_{n}\left( 0\right) $ for $n\in \mathbb{N}$ and $\overline{%
I_{n}}:=G_{m}$ $\backslash $ $I_{n}$ $.$

Let
$
e_{n}:=\left( 0,\dots,0,x_{n}=1,0,\dots\right) \in G_{m}\qquad \left( n\in 
\mathbb{N}\right).
$
If we define the so-called generalized number system based on $m$ in the
following way: 
\begin{equation*}
M_{0}:=1,\text{ \qquad }M_{k+1}:=m_{k}M_{k\text{ }},\ \qquad (k\in \mathbb{N})
\end{equation*}%
then every $n\in \mathbb{N}$ can be uniquely expressed as $%
n=\sum_{k=0}^{\infty }n_{j}M_{j}$, where $n_{j}\in Z_{m_{j}}$ $~(j\in \mathbb{%
N})$ and only a finite number of $n_{j}`$s differ from zero. Let $\left\vert
n\right\vert :=\max $ $\{j\in \mathbb{N};$ $n_{j}\neq 0\}.$

If we define $I_n:=I_n\left(0\right),$ for  $n\in\mathbb{N}$ and $\overline{I_n}:=G_m \ \ \backslash $ $I_n$ and
\begin{equation*}
I_N^{k,l}:=\left\{ \begin{array}{l}I_N(0,\ldots,0,x_k\neq 0,0,...,0,x_l\neq 0,x_{l+1},\ldots ,x_{N-1},\ldots),\\
\text{for} \qquad k<l<N,\\
I_N(0,\ldots,0,x_k\neq 0,x_{k+1}=0,\ldots,x_{N-1}=0,x_N,\ldots ), \\
\text{for } \qquad l=N. \end{array}\right.
\end{equation*}
then
\begin{equation}\label{2}
\overline{I_{N}}=\left( \overset{N-2}{\underset{k=0}{\bigcup }}\overset{N-1%
}{\underset{l=k+1}{\bigcup }}I_{N}^{k,l}\right) \bigcup \left(
\underset{k=0}{\bigcup\limits^{N-1}}I_{N}^{k,N}\right) .  
\end{equation}

Next, we introduce on $G_{m}$ an orthonormal system which is called the
Vilenkin system. First define the complex valued function $r_{k}\left( x\right)
:G_{m}\rightarrow \mathbb{C},$ the generalized Rademacher functions, as 
\begin{equation*}
r_{k}\left( x\right):=\exp \left( 2\pi\imath x_{k}/m_{k}\right) \text{
\qquad }\left( \imath^{2}=-1,\text{ }x\in G_{m},\text{ }k\in \mathbb{N}%
\right).
\end{equation*}

Now define the Vilenkin system $\psi:=(\psi _{n}:n\in \mathbb{N})$ on $G_{m} 
$ as: 
\begin{equation*}
\psi _{n}\left( x\right):=\prod_{k=0}^{\infty }r_{k}^{n_{k}}\left( x\right) 
\text{ \qquad }\left( n\in \mathbb{N}\right).
\end{equation*}

By a Vilenkin polynomial we mean a finite linear
combination of Vilenkin functions. We denote the collection of Vilenkin polynomials by $\mathcal{P}$.

The Vilenkin system is orthonormal and complete in $L_{2}\left( G_{m}\right)
\,$ (for details see e.g. \cite{AVD, sws, Vi}). Specially, we call this system the Walsh-Paley one if $m\equiv 2$ (for details see \cite{gol} and \cite{sws}).

Recall that  (for details see e.g. \cite{AVD}, \cite{gat1} and \cite{gat})
if $n>t,$ $t,n\in \mathbb{N},$ then 
\begin{equation}\label{star1}
K^\psi_{M_n}\left(x\right)=\left\{ \begin{array}{ll}
\frac{M_t}{1-r_t\left(x\right) },& x\in I_t\backslash I_{t+1},\qquad x-x_te_t\in I_n, \\
\frac{M_n+1}{2}, & x\in I_n, \\
0, & \text{otherwise} \end{array} \right.
\end{equation}
and
\begin{equation} \label{star2}
n\left\vert K^\psi_n\right\vert\leq
c\sum_{l=0}^{\left\vert n\right\vert }M_l \left\vert K^\psi_{M_l} \right\vert.
\end{equation}%
By using these two properties of Fejér kernels we obtain the following:

\begin{lemma}\label{lemma7kn}
	For any $n,N\in \mathbb{N_+}$, we have that
	\begin{eqnarray} \label{fn40}
	&&\int_{G_m} K^\psi_n (x)d\mu(x)=1,\\
	&& \label{fn4}
	\sup_{n\in\mathbb{N}}\int_{G_m}\left\vert K^\psi_n(x)\right\vert d\mu(x)\leq c<\infty,\\
	&& \label{fn400}
	\int_{\overline{I_N}}\left\vert K^\psi_n(x)\right\vert d\mu (x)\rightarrow  0, \ \ \text{as} \ \ n\rightarrow  \infty, \ \ \text{for any} \ \ N\in \mathbb{N_+}.
	\end{eqnarray}
	where $c$ is an absolute constant.
\end{lemma}
\begin{proof}
	According orthonormality of Vilenkin systems we immediately get the proof of \eqref{fn40}. It is easy to prove that $$\int_{G_m}\left\vert K^\psi_{M_n}(x)\right\vert d\mu(x)\leq c<\infty.$$
	By combining \eqref{star1} and \eqref{star2} we can conclude that
	\begin{eqnarray*}
		\int_{G_m}\left\vert K^\psi_n\left(x\right)\right\vert d\mu\left(x\right)
		&\leq& \frac{1}{n}\sum_{l=0}^{\left\vert n\right\vert} M_{l}\int_{G_{m}}\left\vert K^\psi_{M_l}\left(x\right)\right\vert d\mu
		\left(x\right) \leq  \frac{1}{n}\sum_{l=0}^{\left\vert n\right\vert}M_l<c<\infty,
	\end{eqnarray*}
	so also \eqref{fn4} is proved.
	
	Let $x\in I_N^{k,l}, \ \ k=0,\dots,N-2, \ \ l=k+1,\dots ,N-1.$ By using again \eqref{star1} and \ref{star2} we get that
	\begin{eqnarray} \label{kn}
	\left\vert K^\psi_n\left(x\right)\right\vert \leq \frac{c}{n}\sum_{s=0}^{l}M_s\left\vert K^\psi_{M_s}\left(x\right)\right\vert 
	\leq \frac{c}{n}\sum_{s=0}^{l}M_sM_k\leq \frac{cM_lM_k}{n}. 
	\end{eqnarray}
	
	Let $x\in I_N^{k,N}$, where $x\in I_{q+1}^{k,q}$, for some $N\leq q <\vert n\vert,$ i.e.,
	$$x=\left(x_0=0,\ldots,x_{k-1}=0,x_k\neq 0,\ldots,x_{N-1}=0,x_q\neq 0,x_{q+1}=0,\ldots,x_{\left\vert n\right\vert-1},\ldots\right), $$
	then
	\begin{eqnarray} \label{11110T1}
	\left\vert K^\psi_n\left(x\right)\right\vert
	\leq\frac{c}{n}\underset{i=0}{\overset{q-1}{\sum}} M_iM_k\leq\frac{cM_kM_q}{n}.
	\end{eqnarray}
	
	Let  $x\in I_{\vert n\vert}^{k,\vert n\vert}\subset I_N^{k,N}$, i.e., $$x=\left(x_0=0,\ldots,x_{m-1}=0,x_k\neq 0,x_{k+1}=0,\ldots,x_N=0,\ldots, x_{\left\vert n\right\vert-1}=0,\ldots\right),$$ then
	\begin{eqnarray} \label{11111T1}
	&&\left\vert K^\psi_n\left(x\right)\right\vert
	\leq \frac{c}{n}\overset{\left\vert n\right\vert-1} {\underset{i=0}{\sum }}M_iM_k
	\leq \frac{cM_kM_{\left\vert n\right\vert}}{n}.
	\end{eqnarray}
	
	If we combine \eqref{11110T1} and \eqref{11111T1} we can conclude that
	\begin{eqnarray}\label{11110T111}
	\int_{I_{N}^{k,N}}\left\vert K^\psi_n \right\vert d\mu
	&=&\overset{\vert n\vert-1}{\underset{q=N}{\sum }}
	\int_{I_{q+1}^{k,q}}\left\vert K^\psi_n \right\vert d\mu
	+\int_{I_{\vert n\vert}^{k,\vert n\vert}}\left\vert K^\psi_n \right\vert d\mu \\ \notag
	&\leq&\overset{\vert n\vert-1}{\underset{q=N}{\sum }}\frac{cM_k}{n}+\frac{cM_k}{n}\\
	&\leq& \frac{c(\vert n\vert-N)M_k}{M_{\vert n \vert}}.
	\end{eqnarray}
	
	Hence, if we apply \eqref{2}, \eqref{kn} and \eqref{11110T111} we find that
	\begin{eqnarray*}
		&&\int_{\overline{I_N}}\left\vert K^\psi_n\right\vert d\mu\\
		&=&\overset{N-2}{\underset{k=0}{\sum }}\overset{N-1}{\underset{l=k+1}{\sum }}
		\sum\limits_{x_{j}=0,j\in\{l+1,...,N-1\}}^{m_{j-1}}\int_{I_{N}^{k,l}}\left\vert K^\psi_n \right\vert d\mu 
		+\overset{N-1}{\underset{k=0}{\sum}} \int_{I_{N}^{k,N}}\left\vert K^\psi_n\right\vert d\mu\\
		&\leq& c\overset{N-2}{\underset{k=0}{\sum }}\overset{N-1}{\underset{l=k+1}{\sum }}\frac{m_{l+1}...m_{N-1}}{M_{N}}
		\frac{cM_lM_k}{n}
		+c\overset{N-1}{\underset{k=0}{\sum }}(\vert n\vert-N)M_k\frac{1}{M_{\vert n \vert}}\\
		&:=&I+II.
	\end{eqnarray*}
	It is evident that
	\begin{eqnarray*}
		I&=&\overset{N-2}{\underset{k=0}{\sum }}\overset{N-1}{\underset{l=k+1}{\sum }}
		\frac{M_k}{M_{\vert n \vert}} \leq c\overset{N-2}{\underset{k=0}{\sum }}
		\frac{(N-k)M_k}{M_{\vert n \vert}}\\
		&\leq& c\overset{N-2}{\underset{k=0}{\sum }}
		\frac{\vert n \vert-k}{2^{\vert n \vert-k}}
		= c\overset{N-2}{\underset{k=0}{\sum }}
		\frac{\vert n \vert-k}{2^{(\vert n \vert-k)/2}}\frac{1}{2^{(\vert n \vert-k)/2}}\\
		&\leq& \frac{c}{2^{(\vert n \vert-N)/2}}\overset{N-2}{\underset{k=0}{\sum }}
		\frac{\vert n \vert-k}{2^{(\vert n \vert-k)/2}}\leq \frac{C}{2^{(\vert n \vert-N)/2}}\to 0, \ \ \text{as} \ \ n\to \infty.
	\end{eqnarray*}
	Analogously, we see that
	\begin{equation*}
	II\leq \frac{c(\vert n\vert-N)}{2^{\vert n \vert-N}}\to 0, \ \ \text{as} \ \ n\to \infty,
	\end{equation*}%
	so also \eqref{fn400} holds and the proof is complete.
	
\end{proof}

The next lemma is very important  to prove almost everywhere convergence of Vilenkin-Fej\'er means:

\begin{lemma}\label{lemma7kncc} Let $n\in \mathbb{N}.$ Then	
	\begin{eqnarray*}
		&& \int_{\overline{I_N}}\sup_{n>M_N}\left\vert K^\psi_n\right\vert d\mu\leq C<\infty,\notag
	\end{eqnarray*}
	where $C$ is an absolute constant.
\end{lemma}

\begin{proof}
	Let $n>M_N$ and $x\in I_N^{k,l}, \ \ k=0,\dots,N-2, \ \ l=k+1,\dots ,N-1.$ By using  \eqref{kn} in the proof of Lemma \ref{lemma7kn} we get that
	\begin{eqnarray} \label{knmax1}
	\sup_{n>M_N}\left\vert K^\psi_n\left(x\right)\right\vert\leq \frac{cM_lM_k}{M_N}.
	\end{eqnarray}
	Let $n>M_N$ and $x\in I_N^{k,N}$. Then, by using \eqref{star1} we find that
$ \left\vert K^\psi_n\left(x\right)\right\vert
	\leq cM_k$
	so that
	\begin{eqnarray} \label{knmax2}
	\sup_{n>M_N}\left\vert K^\psi_n (x)\right\vert
	\leq cM_k.
	\end{eqnarray}
	Hence, if we apply \eqref{2} we get that
	\begin{eqnarray} \label{knmax3}
	&&\int_{\overline{I_N}}\sup_{n>M_N}\left\vert K^\psi_n\right\vert d\mu\\ \notag
	&=&\overset{N-2}{\underset{k=0}{\sum }}\overset{N-1}{\underset{l=k+1}{\sum }}
	\sum\limits_{x_{j}=0,j\in\{l+1,...,N-1\}}^{m_{j-1}}\int_{I_{N}^{k,l}}\sup_{n>M_N}\left\vert K^\psi_n \right\vert d\mu \\ \notag
	&+&\overset{N-1}{\underset{k=0}{\sum }}\int_{I_{N}^{k,N}}\sup_{n>M_N}\left\vert K^\psi_n \right\vert d\mu\\ \notag
	&\leq& c\overset{N-2}{\underset{k=0}{\sum }}\overset{N-1}{\underset{l=k+1}{\sum }}\frac{m_{l+1}...m_{N-1}}{M_{N}}
	\frac{M_lM_k}{M_N}+c\overset{N-1}{\underset{k=0}{\sum }}\frac{M_k}{M_N}\\ \notag
	&\leq& \overset{N-2}{\underset{k=0}{\sum }}
	\frac{(N-k)M_k}{M_N}+c<C<\infty.
	\end{eqnarray}
	
	The proof is complete.	
\end{proof}

\section{Fejér means with respect to the trigonometric  system }

If we consider the  Fej\'{e}r kernels with respect to the trigonometric system $\{(1/2\pi){e^{inx}},n=0,\pm1,\pm 2,...\},$ for $x\in[-\pi,\pi]$ we have that  $K^T_n (x)\geq 0$ and 
$$
K^T_n(x)=\frac{1}{n}\left(\frac{\sin((n x)/2)}{\sin(x/2)}\right)^2.
$$
Moreover, Fejér kernel  $K^T_n  (n\in \mathbb{N_+} )$ with respect to trigonometric system has  upper envelope 
\begin{eqnarray}\label{trkn}
0\leq K^T_n(x)\leq \min(n,\pi{(n \vert x\vert^2)}^{-1}).
\end{eqnarray}
\begin{figure}[htbp]
	\centerline{\includegraphics[scale=0.96]{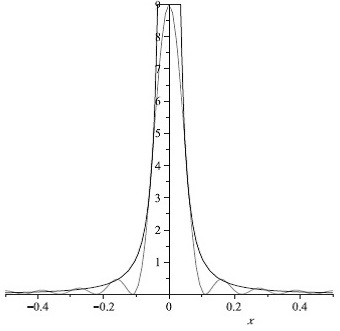}}
	\caption{Fej\'{e}r kernel and the upper envelope $\min(n,\pi{(n \vert x\vert^2)}^{-1})$}
	\label{f1}
\end{figure}

It also follows that every Fejér kernels have one integrable upper envelope:
$$\sup_{n\in\mathbb{N}}K^T_n(x)\leq \pi \vert x\vert^{-2}.$$
	
\begin{lemma}\label{lemma7kn1}
	Let $n\in \mathbb{N}.$ Then, for any $n,N\in \mathbb{N_+}$, we have that
	\begin{eqnarray} \label{fn401}
	&&\int_{[-\pi,\pi]}\left\vert K^T_n(x)\right\vert d\mu(x)=\int_{[-\pi,\pi]} K^T_n (x)d\mu(x)=1,\\
	&& \label{fn4001}
	\int_{[-\pi,\pi] \backslash [-\varepsilon,\varepsilon]}\left\vert K^T_n(x)\right\vert d\mu (x)\rightarrow  0, \ \ \text{as} \ \ n\rightarrow  \infty,  \ \ \text{for any} \ \ \varepsilon>0.
	\end{eqnarray}
	Moreover,
	\begin{eqnarray}\label{1.73app4}
	\lim_{n\to \infty}\sup_{[-\pi,\pi] \backslash [-\varepsilon,\varepsilon]}\left\vert K^T_n(x)\right\vert =0, \ \ \text{for any} \ \ \varepsilon>0,
	\end{eqnarray}
\end{lemma}
\begin{proof}
	According the property $K^T_n (x)\geq 0$ and orthonormality of trigonometric system we immediately get the proof of \eqref{fn401}.  
 On the other hand, \eqref{fn4001} and \eqref{1.73app4} follow estimate \eqref{trkn} so we leave out the details. 
\end{proof}

\section{Approximate Identity}\label{s2,5}

The properties established in Lemma \ref{lemma7kn} and Lemma \ref{lemma7kn1} ensure that kernel of the Fejér means $\{K^w_N\}_{N=1}^\infty \ \left( w=\psi \  \text{or} \  w=T\right),$ with respect to  Vilenkin and trigonometric systems form what is called an approximation identity. To unify the proofs for trigonometric and Vilenkin systems we mean that $I$ denotes $G_m$ or $[-\pi,\pi]$ and  $I_N$ denotes $I_N(0)$ or $[-1/ {2^N},1/{2^N}],$ for  $N\in \mathbb{N_+}.$

\begin{definition}
	The family $\{\Phi_n \}^{\infty}_{n=1}\subset L_{\infty}(I)$ forms an approximate identity provided that
	\begin{eqnarray*}
		&(A1)& \ \ \ \int_{I}\Phi_n(x)d(x)=1
		\label{1.71app}\\
		&(A2)& \ \ \ \sup_{n\in \mathbb{N}}\int_{I}\left\vert \Phi_n(x)\right\vert d\mu(x)<\infty\label{1.72app} \\
		&(A3)& \ \ \ \int_{I \backslash I_N}\left\vert \Phi_n(x)\right\vert d\mu (x)\rightarrow  0, \ \ \text{as} \ \ n\rightarrow  \infty, \ \ \text{for any} \ \ N\in \mathbb{N_+}.
		\label{1.73app}
	\end{eqnarray*}
\end{definition}
The term "approximate identity" is used because of the fact that 
$\Phi_n\ast f \to f \  \text{ as } \  n\to \infty$ in any reasonable sense.

Next, we prove important result, which will be used to obtain norm convergence of some well-known and general summability methods:

\begin{theorem} \label{theoremconv}
	Let $f\in L_p(I),$ where $1\leq p< \infty$ and the family 
	$\{\Phi_n \}^{\infty}_{n=1}\subset L_{\infty}(I)$
	forms an approximate identity. Then
	$$\Vert \Phi_n\ast f - f \Vert_p \to 0 \ \ \text{ as } \ \ n\to \infty.$$
\end{theorem}
\begin{proof} Let $\varepsilon>0.$ By using continuity of $L_p$ norm and  $(A2)$ condition we get that	
	$$\sup_{t\in I_N}\Vert f(x-t)-f(x)\Vert_p\sup_{n\in \mathbb{N}}\Vert\Phi_n\Vert_1<\varepsilon/2.$$	
	By now applying Minkowski's integral inequality and $(A1)$ and $(A3)$ conditions we find that
	\begin{eqnarray*}
		\Vert\Phi_n\ast f-f\Vert_p
		&=&\left\Vert \int_{I}\Phi_n(t)(f(x-t)-f(x))d\mu(t)\right\Vert_p\\
		&\leq&\int_{I}\left\vert \Phi_n(t)\right\vert \left\Vert f(x-t)-f(x)\right\Vert_p d\mu(t)\\
		&=&\int_{I_N}\left\vert \Phi_n(t)\right\vert \left\Vert f(x-t)-f(x)\right\Vert_p d\mu(t)\\
		&+&\int_{I\setminus I_N}\left\vert \Phi_n(t)\right\vert \left\Vert f(x-t)-f(x)\right\Vert_p d\mu(t)\\
		&\leq&\sup_{t\in I_N}\Vert f(x-t)-f(x)\Vert_p\sup_{n\in \mathbb{N}}\left\Vert\Phi_n\right\Vert_1\\
		&+&\sup_{t\in I}\left\Vert f(x-t)-f(x)\right\Vert_p\int_{I\setminus I_N}\left\vert \Phi_n(t)\right\vert d\mu(t) < \varepsilon/2+\varepsilon/2<\varepsilon.
	\end{eqnarray*}
	The proof is complete.
\end{proof}

According to Lemma \ref{lemma7kn} and Lemma \ref{lemma7kn1} we immediately get that the following results holds true:
\begin{corollary}
	Let $f\in L_p(I),$ where $1\leq p< \infty.$ Then
	
	\begin{equation*}
	\left\Vert \sigma^w_nf-f\right\Vert_p\to 0, \ \ \text{as} \ \ n\to\infty, \text{ }\ \ \left( w=\psi \ \ \text{or} \ \ w=T \right).
	\end{equation*}
where $\sigma_n^\psi$ and $\sigma_n^T$ are Vilenkin-Fejér and trigonometric-Fejér means, respectively.
\end{corollary}

\begin{theorem}\label{th2}
	Suppose that $f\in L_1(I)$ and that the family 
$\{\Phi_n \}^{\infty}_{n=1}\subset L_{\infty}(I)$ 
forms an approximate identity. In addition, let
	\begin{eqnarray}\label{1.73app40}
		&(A4)& \ \ \ \lim_{n\to \infty}\sup_{I \backslash I_N}\left\vert \Phi_n(x)\right\vert =0, \ \ \text{for any} \ \ N\in \mathbb{N_+},
	\end{eqnarray}
	
	a) If the function $f$ is continuous at $t_0$, then  
	$$\Phi_n\ast f(t_0)\to f(t_0) \ \ \text{ as } \ \ n\to \infty.$$
	
	b) If the functions $\{\Phi_n \}^{\infty}_{n=1}$ are even and the left and right limits $f(t_0-0)$ and $f(t_0+0)$ do exist and are finite, 
	then
	$$\Phi_n\ast f(t_0)  \to L, \ \ \text{ as } \ \ n\to \infty,$$
	where
	\begin{eqnarray}\label{L}
	L=:\frac{f(t_0+0)+f(t_0-0)}{2}.
	\end{eqnarray}
\end{theorem}
\begin{proof}
	It is evident that
	\begin{eqnarray*}
		\left\vert \Phi_n\ast f(t_0)-f(t_0)\right\vert 
		&=& \left\vert\int_{I}\Phi_n(t)(f(t_0-t)-f(t_0))d\mu(t)\right\vert\\
		&\leq&\left\vert\int_{I_N} \Phi_n(t)(f(t_0-t)-f(t_0)) d\mu(t)\right\vert\\
		&+&\left\vert\int_{I\setminus I_N} \Phi_n(t)f(t_0-t) d\mu(t)\right\vert+\left\vert\int_{I\setminus I_N} \Phi_n(t)f(t_0) d\mu(t)\right\vert\\
		&=:&I+II+III.
	\end{eqnarray*}
	Let $f$ be continuous at $t_0$. For any $\varepsilon>0$ there exists $N$ such that
	
	$$I \leq \sup_{t\in I_N}\vert f(t_0+t)-f(t_0))\vert \sup_{n\in \mathbb{N}}\Vert\Phi_n\Vert_1< \varepsilon/2.$$
	Using (A4) condition, we get that
	
	$$II \leq \sup_{t\in I \backslash I_N}\left\vert \Phi_n(t)\right\vert \Vert f\Vert_1 \to 0, \  \ \text{as} \ \ n\to\infty.$$
	We conclude from (A3) that
	$$III \leq \vert f(t_0)\vert \int_{I\setminus I_N} \vert \Phi_n(t)\vert d\mu(t)\to 0, \  \ \text{as} \ \ n\to\infty.$$
	Thus part a) is proved. 
	
	Since functions $\{\Phi_n \}^{\infty}_{n=1}$ are even, for the proof of part b), we first note that
	\begin{eqnarray*}
		&&(\Phi_n\ast f)(t_0)-L\\
		&=&\int_{I} \Phi_n(t)\left(\frac{f(t_0-t)+f(t_0+t)}{2}-\frac{f(t_0-0)+f(t_0+0)}{2}\right) d\mu(t)
	\end{eqnarray*}
	Hence, if we use part a), we immediately get the proof of part b) so the proof is complete.
\end{proof}

\begin{corollary}
	Let $f\in L_1[-\pi,\pi].$ Then
	the following statements holds true:
	
	a) If the function $f$ is continuous at $t_0$, then  $$\sigma^T_n f(t_0)\to f(t_0) \ \ \text{ as } \ \ n\to \infty.$$
	
	b) Let left and right limits $f(t_0-0)$ and $f(t_0+0)$ do exist and are finite.
	Then
	$$\sigma^T_n f(t_0)  \to L\ \ \text{ as } \ \ n\to \infty,$$
where $L$ is defined by \eqref{L}.
\end{corollary}
\begin{remark} \label{prop:divkn}
	Conditions (A4) and \eqref{trkn} do not hold for the Vilenkin-Fejér kernels. Indeed, by using \eqref{star1}, for any  $ k\in\mathbb{N }_+$ and for any 
	$e_0\in I_n{(e_{0})} \subset G_m \backslash I_n, \ (n\in\mathbb{N}_+)$
	  we get that 
	\begin{eqnarray*}
		\vert K^\psi_{M_k}(e_0)\vert=\left\vert \frac{M_0}{1-r_0\left(e_0\right) }\right\vert 
		=\left\vert \frac{M_0}{1-\exp\left(2\pi \imath /m_{0}\right) }\right\vert=\frac{1}{2\sin(\pi /m_{0})} \geq \frac{1}{2},
	\end{eqnarray*}
	so that
	\begin{eqnarray*}
		\lim_{k\to \infty}\sup_{I_n{(e_{0})} \subset G_m \backslash I_n}\left\vert K^\psi_{M_k}(x)\right\vert 
		&\geq&\lim_{k\to \infty}\left\vert K^\psi_{M_k}(e_0)\right\vert\\
		&\geq& \frac{1}{2}>0, \ \text{for any} \ n\in \mathbb{N_+}. 
	\end{eqnarray*}

	Hence (A4) and \eqref{trkn} are not true for the Fejér kernels with respect to the Vilenkin system.
	However, in some publications you can find that some researchers use such an estimate (for details see \cite{IV}).	
	
Moreover, for any $x\in I_{k}\backslash I_{k-1}$  we have 
$$\vert K^\psi_{M_k}(x)\vert= \left\vert \frac{M_{k-1}}{1-\exp\left(2\pi \imath /m_{k-1}\right) }\right\vert=\frac{M_{k-1}}{2\sin(\pi /m_{k-1})} \geq \frac{M_k}{2\pi}$$
and it follows that Fejér kernels with respect to Vilenkin system do not have one integrable upper envelope. In particular the following lower estimate holds:
$$\sup_{n\in\mathbb{N}}\vert K^\psi_n(x)\vert\geq (2\pi \lambda \vert x\vert)^{-1}, \ \ \ \text{where} \ \ \  \lambda:=\sup_{n\in\mathbb{N}}m_n.$$
\end{remark}

This remark shows that there is an essential difference between the Vilenkin-Fejér kernels and Fejér kernels with respect to trigonometric system. Moreover, Theorem \ref{th2} is useless to prove almost everywhere convergence of Vilenkin-Fejér means.

\section{Almost Everywhere Convergence of Vilenkin-Fej\'er Means} 

The next theorem is very important to study almost everywhere convergence of the Vilenkin-Fej\'er means:

\begin{theorem}\label{weaktype1}
	Suppose that the sigma sub-linear operator $V$ is bounded from $L_{p_1}$ to $L_{p_{1}}$ for some $1<p_1\leq \infty $ and
	\begin{equation*}
	\int\limits_{\overline{I}}\left\vert Vf\right\vert d\mu \leq C\left\Vert f\right\Vert_1
	\end{equation*}
	for  $f\in L_1(G_m)$ and Vilenkin interval $I\subset G_m$ which satisfy
	\begin{equation}\label{atom}
	\text{supp} f\subset I,\text{ \ \ \ \ \ \ }\int_{G_m}fd\mu =0.
	\end{equation}
	Then the operator $V$ is of weak-type $\left( 1,1\right) $, i.e.,
	\begin{equation*}
	\underset{y>0}{\sup }y\mu \left(\left\{ Vf>y\right\}\right) \leq \left\Vert f\right\Vert_1.
	\end{equation*}
\end{theorem}
\begin{theorem}\label{atom0}
	Let   $f\in L_1(G_m).$  Then
	\begin{equation*}
	\underset{y>0}{\sup}y\mu\left\{\sigma^{*,\psi}f>y\right\} \leq \left\Vert f\right\Vert_1.
	\end{equation*}	
\end{theorem}
\begin{proof} By Theorem \ref{weaktype1} we obtain that the proof will be complete if we show that
	\begin{equation*}
	\int\limits_{\overline{I}}\left\vert \sigma^{*,\psi}f\right\vert d\mu \leq \Vert f\Vert_1,
	\end{equation*}%
	for every function $f,$ which satisfy conditions in \eqref{atom} where $I$ denotes the support of the function $f.$
	
	Without lost the generality we may assume that $f$ be a function with support $I$ and $\mu \left( I\right) =M_{N}.$ We
	may assume that $I=I_N.$ It is easy to see that 
	\begin{equation*}
	\sigma^{\psi}_nf =\underset{I_N}{\int} K^{\psi}_n(x-t)f(t) d\mu\left(t\right)=0,  \ \ \ \text{ for } \ \ \ n\leq M_{N}.
	\end{equation*}
	Therefore, we can suppose that $n>M_{N}.$ 
	Hence,
	\begin{eqnarray*}
		&&\left\vert \sigma^{*,\psi}f(x)\right\vert \\
		&\leq& \sup_{n\leq M_N}\left\vert\underset{I_N}{\int} K^{\psi}_n(x-t)f(t) d\mu\left(t\right)\right\vert
		+ \sup_{n>M_N}\left\vert\underset{I_N}{\int} K^{\psi}_n(x-t)f(t) d\mu\left(t\right)\right\vert\\
		&=&\sup_{n>M_N}\left\vert\underset{I_N}{\int} K^{\psi}_n(x-t)f(t) d\mu\left(t\right)\right\vert.
	\end{eqnarray*}
	
	Let $t\in I_N$ and $x\in \overline{I_N}.$ Then $x-t\in \overline{I_N}$ and if we apply Lemma \ref{lemma7kncc} we get that
	\begin{eqnarray*}
		\int\limits_{\overline{I_N}}\left\vert \sigma^{*,\psi}f(x)\right\vert d\mu (x)
		&\leq&\int\limits_{\overline{I_N}}{\sup_{n>M_N}}\underset{I_N}{\int}\left\vert K^{\psi}_{n}\left(x-t\right)f(t) \right\vert d\mu\left(t\right)d\mu\left(x\right)\\
		&\leq&\int\limits_{\overline{I_N}}\underset{I_N}{\int}{\sup_{n>M_N}}\left\vert K^{\psi}_{n}\left(x-t\right)f(t) \right\vert d\mu\left(t\right)d\mu\left(x\right)\\
		&\leq& \int\limits_{I_N}\underset{\overline{I_N}}{\int}{\sup_{n>M_N}}\left\vert K^{\psi}_n\left(x-t\right)f(t) \right\vert d\mu\left(x\right)d\mu\left(t\right)\\
		&\leq& \int\limits_{I_N}\left\vert f(t) \right\vert d\mu\left(t\right)\underset{\overline{I_N}}{\int}{\sup_{n>M_N}}\left\vert K^{\psi}_n\left(x-t\right)\right\vert d\mu\left(x\right)\\
		&\leq& \int\limits_{I_N}\left\vert f(t) \right\vert d\mu\left(t\right)\underset{\overline{I_N}}{\int}{\sup_{n>M_N}}\left\vert K^{\psi}_n\left(x\right)\right\vert d\mu\left(x\right)\\
		&=& \left\Vert f \right\Vert_1 \underset{\overline{I_N}}{\int}{\sup_{n>M_N}}\left\vert K^{\psi}_n\left(x\right)\right\vert d\mu\left(x\right)\\
		&\leq& c\left\Vert f \right\Vert_1.
	\end{eqnarray*}
	
	The proof is complete.
\end{proof}

\begin{theorem}\label{theoremae0}
	Let $f\in L_{1}(G_m)$. Then
	\begin{equation*}
	\sigma_n^\psi f\rightarrow f\text{ \ \ a.e., \ \ as \ \ } n\rightarrow \infty.
	\end{equation*}
\end{theorem}
\begin{proof}
	Since
	\begin{equation*}
	S^\psi_{n}P=P,\text{ for every }P\in \mathcal{P}
	\end{equation*}%
	according to regularity of Fejér means we obtain that	
	\begin{equation*}
	\sigma^\psi_{n}P\rightarrow P\ \ \ \ \ \text{a.e.,} \text{ \ \ \ \ as  \ \ \ }\ n\rightarrow \infty,
	\end{equation*}%
	where $P\in \mathcal{P}$ is dense in the space $L_1$ (for details see e.g. \cite{AVD}).
	
	On the other hand, by using Theorem \ref{atom0}  we obtain that the maximal operator $\sigma^{\ast}$  is bounded from the space $L_1$ to the space $weak-L_{1},$ that is,
	\begin{equation*}
	\sup_{y>0}y \mu \left\{x\in G_m:\left\vert \sigma^{\ast,\psi} f\left(x\right)\right\vert >y \right\} \leq \left\Vert f\right\Vert_1.
	\end{equation*}
	According to  the usual density argument (see Marcinkiewicz and Zygmund \cite{MA}) imply we obtain almost everywhere convergence of Fejér means
	\begin{equation*}
	\sigma^\psi_nf\rightarrow f\text{ \ \ a.e., \ \ as \ \ }n\rightarrow \infty.
	\end{equation*}
	
	The proof is complete.
\end{proof}

\textbf{Acknowledgment: }The author would like to thank the referee for
helpful suggestions.

\end{document}